\documentclass[reqno]{amsart}
\usepackage{amssymb}
\usepackage{ifpdf}
\ifpdf
  \usepackage[hyperindex,pagebackref]{hyperref}
\else
  \expandafter\ifx\csname dvipdfm\endcsname\relax
    \usepackage[hypertex,hyperindex,pagebackref]{hyperref}
  \else
    \usepackage[dvipdfm,hyperindex,pagebackref]{hyperref}
  \fi
\fi
\theoremstyle{plain}
\newtheorem{thm}{Theorem}[section]
\newtheorem{lem}{Lemma}[section]
\numberwithin{equation}{section}
\DeclareMathOperator{\td}{d\mspace{-2mu}}
\allowdisplaybreaks[4]

\begin{document}

\title[Two sharp inequalities for bounding Seiffert mean]
{Two sharp inequalities for bounding the Seiffert mean by the arithmetic, centroidal, and contra-harmonic means}

\author[W.-D. Jiang]{Wei-Dong Jiang}
\address[Jiang]{Department of Information Engineering, Weihai Vocational University, Weihai City, Shandong Province, 264210, China}
\email{\href{mailto: W.-D. Jiang <jackjwd@163.com>}{jackjwd@163.com}}

\author[J. Cao]{Jian Cao}
\address[Cao]{Department of Mathematics, Hangzhou Normal University, Hangzhou City, Zhejiang Province, 310036, China}
\email{\href{mailto: J. Cao <21caojian@gmail.com>}{21caojian@gmail.com}, \href{mailto: J.Cao <21caojian@163.com>}{21caojian@163.com}}

\author[F. Qi]{Feng Qi}
\address[Qi]{School of Mathematics and Informatics\\ Henan Polytechnic University\\ Jiaozuo City, Henan Province, 454010\\ China; Department of Mathematics\\ School of Science\\ Tianjin Polytechnic University\\ Tianjin City, 300387\\ China}
\email{\href{mailto: F. Qi <qifeng618@gmail.com>}{qifeng618@gmail.com}, \href{mailto: F. Qi <qifeng618@hotmail.com>}{qifeng618@hotmail.com}, \href{mailto: F. Qi <qifeng618@qq.com>}{qifeng618@qq.com}}
\urladdr{\url{http://qifeng618.wordpress.com}}

\subjclass[2010]{Primary 26E60; Secondary 11H60, 26A48, 26D05, 33B10}

\keywords{Seiffert mean; Arithmetic mean; Centroidal mean; Contra-harmonic mean; Inequality; Best constant}

\thanks{The first author was partially supported by the Project of Shandong Province Higher Educational Science and Technology Program under grant No. J11LA57}

\begin{abstract}
In the paper, the authors find the best possible constants appeared in two inequalities for bounding the Seiffert mean by the linear combinations of the arithmetic, centroidal, and contra-harmonic means.
\end{abstract}

\maketitle

\section{Introduction}

For $a,b>0$ with $a\ne b$, the Seiffert mean $T(a,b)$ and the centroidal mean $\overline{C}(a,b)$ are defined respectively by
\begin{equation}\label{seiffert-eq1.1}
    T(a,b)=\frac{a-b}{2\arctan \bigl(\frac{a-b}{a+b}\bigr)}
\end{equation}
and
\begin{equation}\label{seiffert-eq1.2}
  \overline{C}(a,b)=\frac{2\bigl(a^2+ab+b^2\bigr)}{3(a+b)}.
\end{equation}
It is well known that
\begin{gather*}
A(a,b)=\frac{a+b}2,\quad   G(a,b)=\sqrt{ab}\,,\quad   S(a,b)=\sqrt{\frac{a^2+b^2}2}\,,\\
C(a,b)=\frac{a^2+b^2}{a+b},\quad  M_p(a,b)=\sqrt[\leftroot{-2}\uproot{2}p]{\frac{a^p+b^p}2}
\end{gather*}
for $p\ne0$ are respectively the arithmetic, geometric, root-square, contra-harmonic and $p$-th power means of two positive numbers $a$ and $b$, that the $p$-th power means $M_{p}(a,b)$ is continuous and strictly increasing with respect to $p\in \mathbb{R}$ for fixed $a,b>0$ with $a\ne b$, and that the inequalities in
\begin{multline}
G(a,b)=M_0(a,b)<A(a,b)=M_1(a,b)<\overline{C}(a,b)\\*
<S(a,b)=M_2(a,b)<C(a,b)
\end{multline}
hold for $a,b>0$ with $a\ne b$. For more information on results of mean values, please refer to, for example,~\cite{emv-log-convex-simple.tex, Guo-Qi-Filomat-2011-May-12.tex, abstract, royal-98, Cheung-Qi-Rev.tex} and closely related references therein.
\par
In~\cite{seiffert13}, Seiffert proved the double inequality
\begin{equation}
A(a,b)=M_1(a,b)<T(a,b)<M_2(a,b)=S(a,b)
\end{equation}
for $a,b>0$ with $a\ne b$.
In~\cite{seiffert14}, H\"{a}st\"{o} showed that the function $\frac{T(1,x)}{M_p(1,x)}$ is increasing with respect to $x\in(0,\infty)$ if $p\le 1$.
In~\cite{seiffert15, seiffert3}, the authors demonstrated that the double inequalities
\begin{equation}
\alpha_1S(a,b)+(1-\alpha_1)A(a,b)<T(a,b)<\beta_1S(a,b)+(1-\beta_1)A(a,b)
\end{equation}
and
\begin{multline}
C\bigl(\alpha_2 a+(1-\alpha_2)b,\alpha_2 b+(1-\alpha_2)a\bigr)<T(a,b)\\*
<C\bigl(\beta_2 a+(1-\beta_2)b,\beta_2 b+(1-\beta_2)a\bigr)
\end{multline}
hold for $a,b>0$ with $a\ne b$ if and only if
\begin{align}
\alpha_1&\le \frac{4-\pi}{\bigl(\sqrt{2}\,-1\bigr)\pi}, & \beta_1&\ge \frac{2}{3}, &
\alpha_2&\le \frac12\Biggl(1+\sqrt{\frac4\pi-1}\,\Biggr), & \beta_2&\ge \frac{3+\sqrt{3}\,}6.
\end{align}
For more information on this topic, please refer to recently published papers~\cite{seiffert1, seiffert10, seiffert11, seiffert12, seiffert4, seiffert5, seiffert6, seiffert7, background-Jiang-Qi.tex, seiffert2, seiffert8, seiffert9} and cited references therein.
\par
For positive numbers $a, b>0$  with $a\ne b$, let
\begin{equation}\label{J(x)-dfn-eq}
J(x)=\overline{C}\bigl(xa+(1-x)b,xb+(1-x)a\bigr)
\end{equation}
on $\bigl[\frac{1}{2},1\bigr]$. It is not difficult to directly verify that $J(x)$ is continuous and strictly increasing on $\bigl[\frac{1}{2},1\bigr]$ and to notice that
\begin{equation}
J\biggl(\frac12\biggr)=A(a,b)<T(a,b) \quad\text{and}\quad J(1)=\overline{C}(a,b)>T(a,b).
\end{equation}
Therefore, it is much natural to ask a question: What are the best constants $\alpha\ge\frac12$ and $\beta\le1$ such that the double inequality
\begin{equation}\label{seiffert-eq1.4}
\overline{C}\bigl(\alpha a+(1-\alpha)b,\alpha b+(1-\alpha)a\bigr)<T(a,b)
<\overline{C}\bigl(\beta a+(1-\beta)b,\beta b+(1-\beta)a\bigr)
\end{equation}
holds for $a,b>0$ with $a\ne b$?
\par
The following Theorem~\ref{seiffert-th1.1}, the first main result of this paper, gives an affirmative answer to this question.

\begin{thm}\label{seiffert-th1.1}
For positive numbers $a,b>0$ with $a\ne b$, the double inequality~\eqref{seiffert-eq1.4} is valid if and only if
\begin{equation}\label{seiffert-th1.1-constant}
\alpha\le \frac{1}{2}\Biggl(1+\sqrt{\frac{12}{\pi}-3}\,\Biggr)\quad \text{and}\quad \beta=1.
\end{equation}
\end{thm}

In~\cite{seiffert16} the author posed an unsolved problem: Find the greatest value $\alpha_1$ and the least value $\beta_1$ such that the double inequality
\begin{equation}\label{seiffert-eq1.3}
    \alpha_1 C(a,b)+(1-\alpha_1)A(a,b)<T(a,b)<\beta_1 C(a,b)+(1-\beta_1)A(a,b)
\end{equation}
holds for $a,b>0$ with $a\ne b$.
\par
The following Theorem~\ref{seiffert-th1.2}, the second main result of this paper, solves this problem.

\begin{thm}\label{seiffert-th1.2}
for $a,b>0$ with $a\ne b$, the double inequality~\eqref{seiffert-eq1.3} holds if and only if $\alpha_1\le \frac{4}{\pi}-1$ and $\beta_1\ge \frac{1}{3}$.
\end{thm}

\section{Proof of Theorem~\ref{seiffert-th1.1}}

In this section, we supply a proof of Theorem~\ref{seiffert-th1.1}.
\par
For simplicity, we denote two numbers in~\eqref{seiffert-th1.1-constant} by $\lambda$ and $\mu$ respectively.
\par
It is clear that, in order to prove the double inequality~\eqref{seiffert-eq1.4}, it suffices to show
\begin{equation}\label{seiffert-eq2.1}
    T(a,b)>\overline{C}\bigl(\lambda a+(1-\lambda)b,\lambda b+(1-\lambda)a\bigr)
\end{equation}
and
\begin{equation}\label{seiffert-eq2.2}
    T(a,b)<\overline{C}\bigl(\mu a+(1-\mu)b,\mu b+(1-\mu)a\bigr).
\end{equation}
From definitions~\eqref{seiffert-eq1.1} and~\eqref{seiffert-eq1.2} we see that both $T(a,b)$ and $\overline{C}(a,b)$ are symmetric and homogenous of degree $1$. Hence, without loss of generality, we assume that $a>b$. If replacing $\frac{a}b>1$ by $t>1$ and letting $p\in \bigl(\frac12,1\bigr)$, then
\begin{multline}\label{seiffert-eq2.3}
\overline{C}\bigl(pa+(1-p)b,pb+(1-p)a\bigr)-T(a,b)\\*
=\frac{[pt+(1-p)]^2+[pt+(1-p)][p+(1-p)t]+[p+(1-p)t]^2}{6(1+t)\arctan\frac{t-1}{t+1}}bf(t),
\end{multline}
where
\begin{equation}\label{seiffert-eq2.4}
\begin{split}
f(t)&=4\arctan \frac{t-1}{t+1}\\*
&\quad-\frac{3(t^2-1)}{[pt+(1-p)]^2+[pt+(1-p)][p+(1-p)t]+[p+(1-p)t]^2}.
\end{split}
\end{equation}
Standard computations lead to
\begin{align}\label{seiffert-eq2.5}
    f(1)&=0,\\\label{seiffert-eq2.6}
    \lim_{t\to\infty}f(t)&=\pi-\frac{3}{p^2-p+1},
\end{align}
and
\begin{equation}\label{seiffert-eq2.7}
    f'(t)=\frac{f_1(t)}{h_1(t)},
\end{equation}
where
\begin{align}
\begin{split}\label{seiffert-eq2.8}
    f_1(t)&=\bigl(4p^4-8p^3+18p^2-14p+1\bigr)t^4-4\bigl(4p^4-8p^3+9p^2-5p+1\bigr)t^3\\
    &\quad+6\bigl(4p^4-8p^3+6p^2-2p+1\bigr)t^2-4\bigl(4p^4-8p^3+9p^2-5p+1\bigr)t\\
    &\quad+4p^4-8p^3+18p^2-14p+1,
\end{split}\\
    f_1(1)&=0,\label{seiffert-eq2.9}
\end{align}
and
\begin{equation*}
h_1(t)=\bigl\{[pt+(1-p)]^2+[pt+(1-p)][p+(1-p)t]+[p+(1-p)t]^2\bigr\}^2\bigl(1+t^2\bigr).
\end{equation*}
Let
\begin{equation*}
f_2(t)=\frac{f_1'(t)}4,\quad f_3(t)=\frac{f_2'(t)}3,\quad\text{and}\quad f_4(t)=\frac{f_3'(t)}2.
\end{equation*}
Then, by standard argument, we have
\begin{align}\label{seiffert-eq2.10}
\begin{split}
    f_2(t)&=\bigl(4p^4-8p^3+18p^2-14p+1\bigr)t^3-3\bigl(4p^4-8p^3+9p^2-5p+1\bigr)t^2\\
    &\quad+3\bigl(4p^4-8p^3+6p^2-2p+1\bigr)t-\bigl(4p^4-8p^3+9p^2-5p+1\bigr),
\end{split}\\\label{seiffert-eq2.11}
    f_2(1)&=0,\\\label{seiffert-eq2.12}
\begin{split}
    f_3(t)&=\bigl(4p^4-8p^3+18p^2-14p+1\bigr)t^2-2\bigl(4p^4-8p^3+9p^2-5p+1\bigr)t\\
    &\quad+4p^4-8p^3+6p^2-2p+1,
\end{split}\\\label{seiffert-eq2.13}
    f_3(1)&=6p^2-6p,\\\label{seiffert-eq2.14}
    f_4(t)&=\bigl(4p^4-8p^3+18p^2-14p+1\bigr)t-\bigl(4p^4-8p^3+9p^2-5p+1\bigr),\\\label{seiffert-eq2.15}
    f_4(1)&=9p^2-9p.
\end{align}
\par
If $p=\lambda$, then the quantities~\eqref{seiffert-eq2.6}, \eqref{seiffert-eq2.13}, and~\eqref{seiffert-eq2.15} become
\begin{gather}\label{seiffert-eq2.16}
    \lim_{t\to\infty} f(t)=0,\\\label{seiffert-eq2.17}
    f_3(1)=\frac{18}{\pi}-6<0,\\\label{seiffert-eq2.18}
    f_4(1)=\frac{27}{\pi}-9<0,
\end{gather}
and
\begin{equation}\label{seiffert-eq2.19}
    4p^4-8p^3+18p^2-14p+1=\frac{36+18\pi-9\pi^2}{\pi^2}>0.
\end{equation}
Thus, from~\eqref{seiffert-eq2.8}, \eqref{seiffert-eq2.10}, \eqref{seiffert-eq2.12}, \eqref{seiffert-eq2.14}, and~\eqref{seiffert-eq2.19}, it is very easy to obtain that
\begin{align}\label{seiffert-eq2.20}
    \lim_{t\to\infty} f_1(t)&=\infty,\\\label{seiffert-eq2.21}
    \lim_{t\to\infty} f_2(t)&=\infty,\\\label{seiffert-eq2.22}
    \lim_{t\to\infty} f_3(t)&=\infty,\\\label{seiffert-eq2.23}
    \lim_{t\to\infty} f_4(t)&=\infty.
\end{align}
From~\eqref{seiffert-eq2.14} and~\eqref{seiffert-eq2.19}, it is clear that the function $f_4(t)$ is strictly increasing on $[1,\infty)$, and so, by virtue of~\eqref{seiffert-eq2.18} and~\eqref{seiffert-eq2.23}, there exists a point $t_0>1$ such that $f_4(t)<0$ on $[1,t_0)$ and $f_4(t)>0$ on $(t_0,\infty)$. Hence, the function $f_3(t) $ is strictly decreasing on $[1,t_0]$ and strictly increasing on $[t_0,\infty)$. Similarly, by~\eqref{seiffert-eq2.17} and~\eqref{seiffert-eq2.22}, there exists a point $t_1>t_0>1$ such that $f_2(t)$ is strictly decreasing on $[1,t_1]$ and strictly increasing on $[t_1,\infty)$, and, by~\eqref{seiffert-eq2.11} and~\eqref{seiffert-eq2.21}, there exists a point $t_2>t_1>1$ such that $f_1(t)$ is strictly decreasing on $[1,t_2]$ and strictly increasing on $[t_2,\infty)$. Further, by~\eqref{seiffert-eq2.7}, \eqref{seiffert-eq2.9}, and~\eqref{seiffert-eq2.20}, there exists a point $t_3>t_2>1$ such that $f(t)$ is strictly decreasing on $[1,t_3]$ and strictly increasing on $[t_3,\infty)$.
Finally, by~\eqref{seiffert-eq2.3} and~\eqref{seiffert-eq2.16}, it is deduced that the function $f(t)$ is negative on $(1,\infty)$. The inequality~\eqref{seiffert-eq2.1} is thus proved.
\par
If $p=\mu=1$, then the function~\eqref{seiffert-eq2.8} becomes
\begin{equation}\label{seiffert-eq2.24}
    f_1(t)=(t-1)^4>0
\end{equation}
for $t>1$. Combining this with~\eqref{seiffert-eq2.7} and~\eqref{seiffert-eq2.5} results in that $f(t)$ is strictly increasing and positive on $(1,\infty)$. Therefore, the inequality~\eqref{seiffert-eq2.2} is obtained.
\par
Combining the inequalities~\eqref{seiffert-eq2.1} and~\eqref{seiffert-eq2.2} with the monotonicity of $J(x)$ defined by~\eqref{J(x)-dfn-eq}, the double inequality~\eqref{seiffert-eq1.4} is established for all  $\alpha\le\lambda$ and $\beta\ge 1$.
\par
For any given number $p$ satisfying $1>p>\lambda$, it is obvious that the limit~\eqref{seiffert-eq2.6} is positive. This positivity together with~\eqref{seiffert-eq2.3} and~\eqref{seiffert-eq2.4} implies that for $1>p>\lambda$ there exists $T_0=T_0(p)>1$ such that the inequality
\begin{equation*}
    \overline{C}\bigl(pa+(1-p)b,pb+(1-p)a\bigr)>T(a,b)
\end{equation*}
holds for $\frac{a}b\in (T_0,\infty)$. This tells us that the constant $\lambda$ is the best possible.
\par
For $\frac12<p<\mu=1$, the quantity~\eqref{seiffert-eq2.13} is positive. Accordingly, there exists a number $\delta=\delta(p)>0$ such that the function $f_3(t)$ is negative on $(1,1+\delta)$. This negativity together with~\eqref{seiffert-eq2.3}, \eqref{seiffert-eq2.5}, \eqref{seiffert-eq2.7} and~\eqref{seiffert-eq2.9} implies that for any  $\frac12<p<\mu=1$, there exists $\delta=\delta(p)>0$ such that the inequality
\begin{equation*}
    T(a,b)>\overline{C}\bigl(pa+(1-p)b,pb+(1-p)a\bigr)
\end{equation*}
is valid for $\frac{a}b\in (T_0,\infty)$. Consequently, the number $\mu$ is the best possible. The proof of Theorem~\ref{seiffert-th1.1} is complete.

\section{Proof of Theorem~\ref{seiffert-th1.2}}

In order to prove Theorem~\ref{seiffert-th1.2}, we need the following Lemmas.

\begin{lem}\label{B-2q-posit}
The Bernoulli numbers $B_{2n}$ for $n\in\mathbb{N}$ have the property
\begin{equation}\label{|B_{2n}|}
(-1)^{n-1}B_{2n}=|B_{2n}|,
\end{equation}
where the Bernoulli numbers $B_i$ for $i\ge0$ are defined by
\begin{equation}
\frac{x}{e^x-1}=\sum_{i=0}^\infty \frac{B_i}{n!}x^i =1-\frac{x}2+\sum_{i=1}^\infty B_{2i}\frac{x^{2i}}{(2i)!}, \quad \vert x\vert <2\pi.
\end{equation}
\end{lem}

\begin{proof}
In~\cite[p.~16 and p.~56]{aar}, it is listed that for $q\ge1$
\begin{equation}\label{zeta(2q)-B(2q)}
\zeta(2q)=(-1)^{q-1}\frac{(2\pi)^{2q}}{(2q)!}\frac{B_{2q}}2,
\end{equation}
where $\zeta$ is the Riemann zeta function defined by
\begin{equation}
  \zeta(s)=\sum_{n=1}^\infty\frac1{n^s}.
\end{equation}
From~\eqref{zeta(2q)-B(2q)}, the formula~\eqref{|B_{2n}|} follows.
\end{proof}

\begin{lem}\label{lem2.0}
For $0<|x|<\pi$,
\begin{equation}\label{eq2.0}
   \cot x=\frac{1}{x}-\sum_{n=1}^{\infty}\frac{2^{2n}|B_{2n}|}{(2n)!}x^{2n-1}.
\end{equation}
\end{lem}

\begin{proof}
This may be derived readily from combining the formula~\cite[p.~75, 4.3.70]{abram} with the identity~\eqref{|B_{2n}|}.
\end{proof}

\begin{lem}\label{lem2.2}
For $0<|x|<\pi$,
\begin{equation}\label{eq2.2}
\frac{1}{\sin^2x}=\frac{1}{x^2} +\sum_{n=1}^{\infty}\frac{2^{2n}(2n-1)|B_{2n}|}{(2n)!}x^{2(n-1)}.
\end{equation}
\end{lem}

\begin{proof}
Since
$$
\frac{1}{\sin ^2x}=\csc^2x=-\frac{\td}{\td x}(\cot x),
$$
the formula~\eqref{eq2.2} follows from differentiating~\eqref{eq2.0}.
\end{proof}

Now we are ready to prove Theorem~\ref{seiffert-th1.2}. It is easy to see that the double inequality~\eqref{seiffert-eq1.3} is equivalent to
\begin{equation}\label{seiffert-eq3.3}
\alpha_1< \frac{T(a,b)-A(a,b)}{C(a,b)-A(a,b)}<\beta_1.
\end{equation}
Without loss of generality, we assume $a>b>0$ and let $x=\frac{a}{b}$. Then $x>1$ and
\begin{equation*}
\frac{T(a,b)-A(a,b)}{C(a,b)-A(a,b)}=\frac{\frac{x-1}{2\arctan \frac{x-1}{x+1}} -\frac{x+1}{2}}{\frac{x^2+1}{x+1}-\frac{x+1}{2}}.
\end{equation*}
Let $t=\frac{x-1}{x+1}$. Then $t\in (0,1)$ and
\begin{equation*}
\frac{T(a,b)-A(a,b)}{C(a,b)-A(a,b)}= \frac{\frac{t}{\arctan t}-1}{t^2}.
\end{equation*}
Let $t=\tan \theta$ for $\theta\in \bigl(0,\frac{\pi}{4}\bigr)$. Then
\begin{equation*}
  \frac{T(a,b)-A(a,b)}{C(a,b)-A(a,b)}=  \frac{\frac{\tan\theta}{\theta}-1}{(\tan\theta)^2}=\frac{\cot \theta}{\theta}-\frac{1}{\sin ^2 \theta}+1.
\end{equation*}
By Lemmas~\ref{lem2.0} and~\ref{lem2.2}, we have
\begin{align*}
\frac{\cot \theta}{\theta}-\frac{1}{\sin ^2 \theta}+1
&=1-\sum_{n=1}^{\infty}\frac{2^{2n}}{(2n)!}|B_{2n}|\theta^{2n-2} -\sum_{n=1}^{\infty}\frac{(2n-1)2^{2n}}{(2n)!}|B_{2n}|\theta^{2n-2}\\
&=1-\sum_{n=1}^{\infty}\frac{n2^{2n+1}}{(2n)!}|B_{2n}|\theta^{2n-2}
\end{align*}
which is strictly decreasing on $\bigl(0,\frac{\pi}{4}\bigr)$. Moreover, by L'H\^{o}spital rule and standard argument, we have
\begin{equation*}
\lim_{x\to 0^+}=\frac{1}{3}\quad \text{and}\quad \lim_{x\to(\pi/4)^-}=\frac{4}{\pi}-1.
\end{equation*}
The proof of Theorem~\ref{seiffert-th1.2} is complete.

\end{document}